\documentclass[draft,reqno]{amsart}
\usepackage[american]{babel}
\usepackage{amssymb,euscript}
\numberwithin{equation}{section}
\newtheorem{theorem}{Theorem}[section]
\newtheorem{lemma}[theorem]{Lemma}

\theoremstyle{definition}
\newtheorem{definition}[theorem]{Definition}
\theoremstyle{remark}
\newtheorem{remark}[theorem]{Remark}

\def\Re{\operatorname{{Re}}}
\def\Im{\operatorname{{Im}}}

\DeclareMathOperator*{\supp}{supp}
\def\sgn{\operatorname{{sign}}}

\author[Faminskii]{Andrei V. Faminskii}
\address{RUDN University, 6 Miklukho-Maklaya Street, Moscow, 117198, Russian Federation}
\email{afaminskii@sci.pfu.edu.ru}
\thanks{This paper has been supported by Russian Science Foundation grant 23-11-00056.}

\title[Higher order nonlinear Schr\"{o}dinger equation]{Global weak solutions of an initial-boundary value problem on a half-line for the higher order nonlinear Schr\"{o}dinger equation}

\date{}

\begin{document}
\maketitle

\begin{abstract}
An initial-boundary value problem with one boundary condition is considered for the higher order nonlinear Schr\"odinger equation. It is assumed that either the boundary condition is homogeneous or the nonlinearity in the equation is quadratic. Results on existence, uniqueness and continuous dependence on input data of global  weak solutions are obtained.
\end{abstract}

\section{Introduction. Notation. Description of main results}\label{S1}

In the paper we consider the  higher order nonlinear Schr\"odinger equation (HNLS)
\begin{equation}\label{1.1}
i u_t  + a u_{xx} + i b u_x + i u_{xxx} + \lambda |u|^p u +
 i \beta \bigl( |u|^p u\bigr)_x + i \gamma \bigl( |u|^p\bigr)_x u= f(t,x),
\end{equation}
posed on the semiaxis $\mathbb R_+ = (0, +\infty)$.
Here $a$, $b$, $\lambda$, $\beta$, $\gamma$ are real constants, $p\geq 1$, $u=u(t,x)$ and $f$ are complex-valued functions (as well as all other functions below, unless otherwise stated).

For an arbitrary $T>0$ in a half-strip $\Pi_T^+ = (0,T) \times \mathbb R_+$ we study for this equation an initial-boundary value problem with an initial condition
\begin{equation}\label{1.2}
u(0,x) = u_0 (x),\quad x \geq 0,
\end{equation}
and a boundary condition
\begin{equation}\label{1.3}
u(t,0) = \mu(t), \quad t\in [0,T].
\end{equation}
We consider two cases: either $\mu\equiv 0$ or $p=1$, $\gamma=0$. In both cases results on existence, uniqueness and continuous dependence on the input data $u_0$, $f$, $\mu$ (in the second case) of global weak solutions are established.

Equation \eqref{1.1} is a generalized combination of the well-known nonlinear Schr\"odinger equation (NLS) in one spatial dimension
$$
i u_t +a u_{xx} + \lambda |u|^p u =0
$$
and the Korteweg--de~Vries equation (KdV)
$$
u_t+bu_x+u_{xxx}+\beta u u_x=0.
$$
It has various physical applications, in particular, it models propagation of femtosecond optical pulses in a monomode optical fiber, accounting for additional effects such as third order dispersion, self-steeping of the pulse, and self-frequency shift (see \cite{Fib, HK, Kod, KC} and the references therein).

Similarly to the NLS and KdV equations it can be shown that the norm in $L_2(\mathbb R)$ of a solution to the initial value problem for the HNLS equation is preserved, that is
$$
\frac{d}{dt} \int_{\mathbb R} |u|^2\, dx =0.
$$
It is well known, that for the NLS and KdV equations the following quantities, which are usually referred to as the energy, are also preserved by the solution flow:
$$
\int_{\mathbb R} \Bigl(|u_x|^2 -\frac{2\lambda}{a(p+2)} |u|^{p+2}\Bigr)\,dx,\quad 
\int_{\mathbb R} \Bigl(u_x^2 -\frac{\beta}{3} u^{3}\Bigr)\,dx.
$$
On the contrary, if one tries to derive the analogue of the conservation law for the energy in the case of the HNLS equation, the following identity is obtained:
\begin{multline*}
\frac{d}{dt} \int_{\mathbb R}\Bigl[ |u_x|^2  + \frac{i}{\beta+\gamma}\Bigl(\lambda - \frac{a(3\beta+2\gamma)}{3}\Bigr) u \bar u_x 
-\frac{2(3\beta +2\gamma)}{3(p+2)} |u|^{p+2}\Bigr]\, dx \\ -
\frac{\gamma}{3} \int_{\mathbb R} \bigl(|u|^p\bigr)_x \bigl(|u|^2\bigr)_{xx} \,dx =0.
\end{multline*}
Therefore, the integral of the expression in the square brackets is preserved only if either $\gamma=0$ or $p=2$,
since
$$
\int_{\mathbb R} \bigl(|u|^2\bigr)_x \bigl(|u|^2\bigr)_{xx} \,dx =0.
$$

Another serious obstacle in the study of such an equation is the non-smoothness of the function $|u|^p$ except the special cases of even natural values of $p$. As a result, just in the case $p=2$, that is for the equation
$$
i u_t + a u_{xx} + ib u_x +i u_{xxx} + \lambda |u|^2 u + i\beta \bigl(|u|^2 u\bigr)_x + i\gamma (|u|^2)_x u =0,
$$
for the first time were obtained results on local and global well-posedness of the initial value problem. In particular, in \cite{Laurey} local well-posedness was proved for the initial data $u_0$ in $H^s(\mathbb R)$, $s>3/4$, and global well-posedness in $H^s(\mathbb R)$, $s\geq 1$, if $\beta+\gamma\ne 0$. In \cite{Staf} the local result was improved up to $s\geq 1/4$. In \cite{CL} the same result was obtained for the analogous equation with variable (depending on $t$) coefficients. Under certain relation between the coefficients the local result was extended to the global one in \cite{Car06} if $s>1/4$. In \cite {CN} global well-posedness was proved in certain weighted subspaces of $H^2(\mathbb R)$, where power weights were effective at $\pm\infty$; the argument used global estimates from \cite{Laurey}. A unique continuation property was obtained in \cite{CP}. Local well-posedness of the periodic initial value problem in $H^{1/2}(\mathbb T)$ was established in \cite{Tak}. For the truncated NHLS equation ($\beta= \gamma =0$) local well posedness in $H^s(\mathbb R)$ for $s>-1/2$ was proved in \cite{Car04}.

For other values of $p$ the initial value problem for the HNLS equation was considered in the recent paper \cite{F23} in the case $p=1$, $\gamma=0$. Results on existence and uniqueness of global weak solutions in certain weighted subspaces of $L_2(\mathbb R)$ for power weights effective at $+\infty$ were established. Moreover in the presence of additional damping, effective at $\pm \infty$, exponential large-time decay of solutions was obtained.

In \cite{CCFSV} an initial-boundary value problem on a bounded interval $I$ for the HNLS equation was considered. In the case $p\in [1,2]$ and the initial function $u_0 \in H^s(I)$, $0\leq s \leq 3$, results on global existence and uniqueness of mild solutions were obtained. For $u_0\in L_2(I)$ the result on global existence was extended either to $p\in (2,3)$ or $p\in (2,4)$, $\gamma=0$. Moreover, after addition to the equation of the damping term $id(x)u$, where the non-negative function $d$ was strictly positive on a certain sub-interval, large-time decay of solutions was established.  Certain preceding results for the truncated HNLS equation ($\beta=\gamma=0$) can be found in \cite{ASV, BOY, BBV, BV, Chen, CCPV}.

The initial-boundary value problem on the semiaxis $\mathbb R_+$ is not studied yet, although it has transparent physical meaning, describing the waves moving from the boundary point. In recent paper \cite{AMO} local well-posedness results were established for the truncated NHLS equation ($\beta= \gamma =0$) in the case of the non-homogeneous boundary data \eqref{1.3}.
The results and methods of that paper are not used here.

\bigskip

Let $L_{q,+} = L_q(\mathbb R_+)$, $H^s_+ = H^s(\mathbb R_+)$, $C_{b,+}^k = C_b^k(\overline{\mathbb R}_+)$, $\EuScript S_+$  be the reduction on $\overline{\mathbb R}_+$ of the Schwartz space $ \EuScript S = \EuScript S(\mathbb R)$. The notation $C_w$ means a weakly continuous map, the subscript $b$ means a bounded map. 

The notion of a weak solution of the considered problem is understood in the following sense.

\begin{definition}\label{D1.1}
Let $u_0\in L_{2,+}$, $\mu\in L_2(0,T)$, $f\in L_1(0,T;L_{2,+})$. A function $u\in L_\infty(0,T;L_{2,+})$, $u_x\in L_2((0,T)\times (0,r))$ $\forall r>0$, is called a weak solution of problem \eqref{1.1}--\eqref{1.3} if for any function $\phi \in C^\infty([0,T]; \EuScript S_+)$, $\phi\big|_{t=T} =0$, $\phi\big|_{x=0} = \phi_x\big|_{x=0} \equiv 0$, the functions $|u|^p u \phi, |u|^p u \phi_x, \gamma|u|^p u_x \phi \in L_1(\Pi_T^+)$ and the following equality holds:
\begin{multline}\label{1.4}
\iint_{\Pi_T^+} \big[ u( i \phi_t - a \phi_{xx} + ib \phi_x  + i \phi_{xxx}) -\lambda |u|^p u \phi  + i\beta |u|^p u \phi_x  + i\gamma |u|^p (u\phi)_x +f\phi \bigr]\,dxdt  \\+
\int_{\mathbb R_+} u_0 \phi\big|_{t=0}\,dx  + i \int_0^T \mu \phi_{xx}\big|_{x=0} \, dt  =0.
\end{multline}
\end{definition}

In fact, the constructed in the paper solutions possess certain additional properties than the ones from this definition. Introduce certain notation.

We say that $\psi(x)$ is an admissible weight function if $\psi$ is an infinitely smooth positive function on $\overline{\mathbb R}_+$, such that 
\begin{equation}\label{1.5}
|\psi^{(j)}(x)|\leq c(j)\psi(x)\quad \text{for each natural\ } j \text{\ and\ } \forall x\geq 0.
\end{equation}
Note that such a function obviously satisfies an inequality 
\begin{equation}\label{1.6}
\psi(x) \leq \psi(0) e^{c(1) x} \quad \forall x\geq 0.
\end{equation}
It was shown in \cite{F12} that $\psi^s$ for any $s\in\mathbb R$ satisfies \eqref{1.5}, so $\psi^s$ is an admissible weight function. Any functions $\psi(x) \equiv e^{2\alpha x}$ as well as $\psi(x) \equiv (1+x)^{2\alpha}$, $\alpha\in \mathbb R$, are admissible weight functions, moreover, $\psi'(x)$ are also admissible weight functions if $\alpha>0$. As an another important example of admissible functions, we define $\rho_0(x)\equiv 1+ \frac{2}{\pi}\arctan x$. Note that both $\rho_0$ and $\rho'_0$ are admissible weight functions.

For an admissible weight function $\psi(x)$ let 
\begin{equation}\label{176}
L_{2,+}^{\psi(x)}= \{\varphi(x): \varphi\psi^{1/2}(x)\in L_{2,+}\}.
\end{equation}
Obviously, $L_{2,+}^{\rho_0(x)}=L_{2,+}$.

Let both $\psi(x)$ and $\psi'(x)$ be admissible weight functions. We construct solutions of the considered problem in spaces $X_w^{\psi(x)}(\Pi_T^+)$, consisting of functions $u(t,x)$ such that 
\begin{equation}\label{1.8}
u\in C_w([0,T]; L_{2,+}^{\psi(x)}), \quad u_x \in L_2(0,T;L_{2,+}^{\psi'(x)}).
\end{equation}
For auxiliary linear results we also use spaces $X^{\psi(x)}(\Pi_T^+)$, where in comparison with $X_w^{\psi(x)}(\Pi_T^+)$ the weak continuity with respect to $t$ in \eqref{1.8} is substituted by the strong one. 

Let also
\begin{equation}\label{1.9}
\sigma^+ (f;T) = \sup\limits_{x_0\geq 0} \Bigl(\int_0^T \int_{x_0}^{x_0+1} |f(t,x)|^2 \,dxdt \Bigr)^{1/2}.
\end{equation}

The main results of the paper are the following two theorems.

\begin{theorem}\label{T1.1}
Let $u_0\in L_{2,+}^{\psi(x)}$, $\mu\equiv 0$, $f\in L_1(0,T;  L_{2,+}^{\psi(x)})$ for certain admissible weight function $\psi(x)$, such that $\psi'(x)$ is also an admissible weight function. Let either $p\in [1,3)$ or $p\in [1,4)$, $\gamma=0$.
Then there exists a weak solution of problem \eqref{1.1}--\eqref{1.3} $u\in X_w^{\psi(x)}(\Pi_T^+)$; moreover $\sigma^+(u_{x};T) <+\infty$.  If, in addition, $p\leq 2$ and for certain positive constant $c_0$
\begin{equation}\label{1.10}
(\psi'(x))^{p+2} \psi^{p-2}(x)\geq c_0\quad \forall x\geq 0,
\end{equation}
then this solution is unique in $X_w^{\psi(x)}(\Pi_T^+)$  and the map
\begin{equation}\label{1.11}
(u_0, f) \to u
\end{equation}
is Lipschitz continuous on any ball in the space $L_{2,+}^{\psi(x)} \times  L_1(0,T;  L_{2,+}^{\psi(x)})$ into the space $ X_w^{\psi(x)}(\Pi_T^+)$.
\end{theorem}

\begin{remark}\label{R1.1}
The exponential weight $\psi(x)\equiv e^{2\alpha x}$ $\forall \alpha>0$ and the power weight $\psi(x)\equiv (1+x)^{2\alpha}$, $\alpha\geq (p+2)/(4p)$, satisfy the hypothesis of the theorem (including uniqueness). If either $p\in [1,3)$ or $p\in [1,4)$ and $\gamma=0$, $u_0\in L_{2,+}$, $f\in L_1(0,T;L_{2,+})$, there exists a weak solution $u\in C_w([0,T];L_{2,+})$, $\sigma^+(u_{x};T) <+\infty$. 
\end{remark}

\begin{theorem}\label{T1.2}
Let $u_0\in L_{2,+}^{\psi(x)}$, $\mu \in H^s(0,T)$, $s>1/3$ , $f\in L_1(0,T;  L_{2,+}^{\psi(x)})$ for certain admissible weight function $\psi(x)$, such that $\psi'(x)$ is also an admissible weight function. Let $p =1$, $\gamma=0$.
Then there exists a weak solution of problem \eqref{1.1}--\eqref{1.3} $u\in X_w^{\psi(x)}(\Pi_T^+)$; moreover $\sigma^+(u_{x};T) <+\infty$.  If, in addition, inequality \eqref{1.10} holds for $p=1$, then this solution is unique in $X_w^{\psi(x)}(\Pi_T^+)$  and the map
\begin{equation}\label{1.12}
(u_0, \mu, f)  \to u
\end{equation}
is Lipschitz continuous on any ball in the space $L_{2,+}^{\psi(x)} \times H^s(0,T) \times  L_1(0,T;  L_{2,+}^{\psi(x)})$ into the space $ X_w^{\psi(x)}(\Pi_T^+)$.
\end{theorem}

\begin{remark}\label{R1.2}
The exponential weight $\psi(x)\equiv e^{2\alpha x}$ $\forall \alpha>0$ and the power weight $\psi(x)\equiv (1+x)^{2\alpha}$, $\alpha\geq 3/4$, satisfy the hypothesis of the theorem (including uniqueness). If $p =1$, $\gamma=0$, $u_0\in L_{2,+}$, $\mu \in H^s(0,T)$ for $s>1/3$, $f\in L_1(0,T;L_{2,+})$, there exists a weak solution $u\in C_w([0,T];L_{2,+})$, $\sigma^+(u_{x};T) <+\infty$. 
\end{remark}

\begin{remark}\label{R1.3}
For $u_0\in L_{2,+}$ the smoothness assumption on the boundary data $\mu \in H^{1/3}(0,T)$ is natural and originates from the properties of the differential operator $\partial_t + \partial_x^3$, so from this point of view the result, obtained in Theorem \ref{T1.2}, is $\varepsilon$-close to natural. The same $\varepsilon$-close to natural result for global weak solutions was previously obtained for the similar initial-boundary value problem for the KdV equation itself in comparison with natural assumptions for the local result (for more details see, for example, \cite{F12} and references therein).  
\end{remark}

\begin{remark}\label{R1.4}
The constraint $p=1$ in the non-homogeneous case originates from the method of the proof of the basic global estimate on the solution in the space $L_{2,+}$ (see, for example, \cite{F12} for more details). For local results this constraint is not used (see \cite{AMO}).
\end{remark}

Further, the symbol $\eta(x)$ denotes a cut-off function, namely, $\eta$ is an infinitely smooth non-decreasing function on $\mathbb R$ such that $\eta(x)=0$ for $x\leq 0$, $\eta(x)=1$ for $x\geq 1$, $\eta(x)+\eta(1-x) \equiv 1$.

Let $\EuScript S_{exp,+} =\EuScript S_{exp}(\overline{\mathbb R}_+)$ be a space of infinitely smooth functions $\varphi(x)$ on $\overline{\mathbb R}_+$,  such that $e^{n x}|\varphi^{(j)}(x)|\leq c(n,j)$ for any non-negative integers $n$, $j$  and $x\in \overline{\mathbb R}_+$.

In what follows we drop the limits in the integrals over $\mathbb R_+$. 

The paper is organized as follows.  Section~\ref{S2} contains certain preliminaries, in particular, on the corresponding auxiliary linear problem and interpolating inequalities. In Section~\ref{S3} the results on existence of solutions to the original problem, and in Section~\ref{S4} --- on uniqueness and continuous dependence are established. 

\section{Preliminaries}\label{S2}

The following interpolating inequality for weighted Sobolev spaces is crucial for the study.

\begin{lemma}\label{L2.1}
Let $\psi_1(x)$, $\psi_2(x)$ be two admissible weight functions, $q\in [2,+\infty]$,
\begin{equation}\label{2.1}
s = s(q) = \frac{1}4 - \frac1{2q}.
\end{equation}
Then for every function $\varphi(x)$ satisfying $\bigl(|\varphi'| +|\varphi|\bigr)\psi_1^{1/2}(x)\in L_{2,+}$, $\varphi\psi_2^{1/2}(x)\in L_{2,+}$, the following inequality holds:
\begin{equation}\label{2.2}
\bigl\| \varphi\psi_1^s \psi_2^{1/2-s}\bigr\|_{L_{q,+}} \leq c
\bigl\|\bigl(|\varphi'| + |\varphi|\bigr)\psi_1^{1/2}\bigr\|^{2s}_{L_{2,+}}
\bigl\|\varphi\psi_2^{1/2}\bigr\|_{L_{2,+}}^{1-2s}, 
\end{equation}
where the constant $c$ depends on $q$ and the properties of the functions $\psi_j$.
\end{lemma}

\begin{proof}
Without loss of generality, assume that $\varphi$ is a smooth decaying at $+\infty$ function (for example, $\varphi\in \EuScript S_{exp,+}$).

First consider the  case $q=+\infty$, then the proof is based on a simple inequality for functions $f\in W_{1,+}^1$
$$
\sup\limits_{x>0} |f(x)| \leq \int |f'(x)|\,dx, 
$$
applied to $f \equiv \varphi^2 \psi_1^{1/2} \psi_2^{1/2}$.
In fact, the properties of admissible weight functions imply that
\begin{multline*}
\sup\limits_{x>0} |\varphi^2 \psi_1^{1/2} \psi_2^{1/2}| \leq 2\int |\varphi \varphi'| \psi_1^{1/2} \psi_2^{1/2} \,dx + c\int |\varphi|^2 \psi_1^{1/2} \psi_2^{1/2} \,dx \\ \leq
c_1\Bigl(\int \bigl( |\varphi'|^2 + |\varphi|^2\bigr) \psi_1 \,dx\Bigr)^{1/2} \Bigl(\int |\varphi|^2 \psi_2 \,dx \Bigr)^{1/2},
\end{multline*}
whence the desired result follows.

If $q<+\infty$ then
\begin{multline*}
\|\varphi \psi_1^s \psi_2^{1/2-s} \|_{L_{q,+}} \leq \Bigl( \sup\limits_{x>0} |\varphi| \psi_1^{1/4} \psi_2^{1/4} \Bigr)^{(q-2)/q} \Bigl( \int |\varphi|^2 \psi_2 \,dx \Bigr)^{1/q} \\ \leq
\Bigl( \int \bigl(|\varphi'|^2 + |\varphi|^2 \bigr)\psi_1 \,dx \Bigr)^{(q-2)/(4q)} \Bigl( \int |\varphi|^2 \psi_2 \,dx \Bigr)^{(q+2)/(4q)}.
\end{multline*}
\end{proof}

Inequality \eqref{2.2} is the generalization for weighted spaces on $\mathbb R_+$ of the well-known interpolating inequality
\begin{equation}\label{2.3}
\|\varphi\|_{L_q(I)} \leq c(q,I) \bigl(\|\varphi'\|^{2s} \|\varphi\|_{L_2(I)}^{1-2s} + \|\varphi\|_{L_2(I)}\bigr),
\end{equation}
where $s$ is given by \eqref{2.1}, valid for $q\in [2,+\infty]$ and any interval $I \subset \mathbb R$.

\bigskip

Consider a linear problem in $\Pi_T^+$
\begin{gather}\label{2.4}
i u_t +a u_{xx} +i b u_x+i u_{xxx}  = f(t,x),\\ 
\label{2.5}
u(0,x) = u_0(x),\quad u(t,0) = \mu(t).
\end{gather}

\begin{lemma}\label{L2.2}
Let $u_0 \in \EuScript S_{exp,+}$, $\mu \in C^\infty[0,T]$, $f\in C^\infty([0,T];\EuScript S_{exp,+})$ and 
\begin{equation}\label{2.6}
\mu^{(l)}(0) = \widetilde\Phi_l(0) \quad \forall l,
\end{equation}
where
\begin{equation}\label{2.7}
\widetilde\Phi_0(x) \equiv u_0(x),\quad \widetilde\Phi_l(x) \equiv -i \partial_t^{l-1} f(0,x) + \bigl( ia \partial_x^2 -b\partial_x - \partial_x^3\bigr) \widetilde\Phi_{l-1}(x), \ l\geq 1.
\end{equation}
Then there exists a unique solution to problem \eqref{2.1}, \eqref{2.2} $u\in C^\infty([0,T];\EuScript S_{exp,+})$.
\end{lemma}

\begin{proof}
Extend the functions $u_0$ and $f$ to the the whole real line such that $u_0 \in \EuScript S$, $f\in C^\infty([0,T]; \EuScript S)$ and consider the initial value problem
\begin{equation}\label{2.8}
i w_t +a w_{xx} +i b w_x+i w_{xxx}  = f(t,x), \quad w(0,x) = u_0(x).
\end{equation}
The unique solution of this problem $w(t,x)$ in the space $C^\infty([0,T]; \EuScript S)$ is constructed via the Fourier transform and is given by a formula
\begin{equation}\label{2.9}
\widehat w(t,\xi) = \widehat u_0(\xi) e^{i(\xi^3 -a \xi^2 -b\xi)t} -i
\int_0^t \widehat f(\tau,\xi) e^{i(\xi^3 -a \xi^2 -b\xi)(t-\tau)}\,d\tau,
\end{equation}
where
$$
\widehat w(t,\xi) \equiv \int_{\mathbb R} e^{-i x\xi} w(t,x)\,dx
$$
with the similar notation for $\widehat u_0$ and $\widehat f$.

Next, let $w_j(t,x) \equiv \partial_x^j w(t,x)$ for some $j$. Then the function $w_j$ is the solution of the problem of \eqref{2.8} type, where $u_0$, $f$ are replaced by $u_0^{(j)}$, $\partial_x^j f$. Let $m\geq 3$, $\psi (x)\equiv x^m$ (note that this function is not an admissible weight one). Multiplying the corresponding equation for $w_j$ by $2\bar w_j(t,x)\psi(x)$, integrating over $\mathbb R_+$ and taking the imaginary part, we derive an equality (here we use that $\psi^{(l)}(0) =0$ for $l\leq 2$)
\begin{multline}\label{2.10}
\frac{d}{dt}\int |w_j|^2\psi\, dx +3 \int |w_{jx}|^2\psi' \,dx = 2a \Im \int w_{jx} \bar w_j \psi'\, dx \\ +
 \int |w_j|^2 \psi'''\, dx + b \int |w_j|^2 \psi' \, dx + 2\Im \int \partial_x^j f \bar w_j \psi\, dx. 
\end{multline} 
Since
$$
\Bigl| 2a \Im \int w_{jx} \bar w_j \psi'\, dx \Bigr| \leq \int |w_{jx}|^2\psi' \,dx + a^2 \int |w_j|^2 \psi' \, dx,
$$
it follows from \eqref{2.10} that
\begin{equation}\label{2.11}
\frac{d}{dt}\int |w_j|^2\psi\, dx \leq \int |w_j|^2 \bigl( (a^2 +|b|) \psi' + \psi''' \bigr)\, dx + 2\Im \int \partial_x^j f \bar w_j \psi\, dx.
\end{equation}
Fix $\alpha>0$ and let $n\geq 3$. For any $m\in [3,n]$ multiplying the corresponding inequality \eqref{2.11} by $(2\alpha)^m/(m!)$ and summing by $m$ we obtain that for
$$
z_n(t) \equiv \int \sum\limits_{m=0}^n \frac{(2\alpha x)^m}{m!} |w_j(t,x)|^2 \,dx
$$
because of the special choice of the functions $\psi$ inequalities
$$
z'_n(t) \leq c z_n(t) +c, \quad z_n(0) \leq c,
$$
hold uniformly with respect to $n$, whence it follows that 
$$
\sup\limits_{t\in [0,T]} \int e^{2\alpha x} |\partial_x^j w(t,x)|^2 \, dx <\infty.
$$
As a result, $w\in C^\infty([0,T]; \EuScript S_{exp,+})$.

Next, let $\mu_0(t) \equiv \mu(t) - w(t,0)$. Note that $\mu_0 \in C^\infty[0,T]$ and $\mu_0^{(l)}(0) =0$ $\forall l$ because of the compatibility conditions \eqref{2.6}, \eqref{2.7}. Consider in $\Pi_T^+$ an initial-boundary value problem
\begin{equation}\label{2.12}
i v_t +a v_{xx} +i b v_x+i v_{xxx}  = 0,\quad v(0,x) =0,\quad v(t,0) = \mu_0(t).
\end{equation}
Let $\Psi(t,x) \equiv \mu_0(t) \eta(1-x)$, $F(t,x) \equiv - (i\Psi_t + a \Psi_{xx} +ib\Psi_x + i\Psi_{xxx})(t,x)$. It is obvious that $F\in C^\infty([0,T]: \EuScript S_{exp,+})$ and $\partial_t^l F(0,x) \equiv 0$ $\forall l$. Problem \eqref{2.12} is equivalent to a problem
\begin{equation}\label{2.13}
i V_t +a V_{xx} +i b V_x +i V_{xxx}  = F(t,x),\quad V(0,x) =0,\quad V(t,0) = 0
\end{equation}
for the function $V(t,x) \equiv v(t,x) - \Psi(t,x)$.

Apply the Galerkin method. Let $\{\varphi_j(x): j=1,2\dots\}$ be a set of linearly independent functions complete in the space $\{\varphi \in H^3_+: \varphi(0) = \varphi'(0) =0\}$. Construct an approximate solution of problem \eqref{2.13} in the form
\begin{equation}\label{2.14}
V_k(t,x) = \sum\limits_{j=1}^k c_{kj}(t) \varphi_j(x)
\end{equation}
via conditions
\begin{gather}\label{2.15}
\int \bigl[ i V_{kt}\bar\varphi_m  +  V_k (a \bar\varphi''_{m} -ib \bar\varphi'_{m} - i\bar\varphi'''_{m}) -F\bar\varphi_m \bigr]\, dx =0, \quad m= 1,\dots,k,\ t\in [0,T]; \\
\label{2.16}
c_{kj}(0) =0.
\end{gather}
Multiplying \eqref{2.15} by $2\bar c_{km}(t)$, summing with respect to $m$ and taking the imaginary part, we find that
\begin{equation}\label{2.17}
\frac{d}{dt} \int |V_k|^2\, dx = 2\Im \int F \bar V_k \,dx,
\end{equation}
whence it follows that
\begin{equation}\label{2.18}
\|V_k\|_{L_\infty(0,T; L_{2,+})} \leq \|F\|_{L_1(0,T;L_{2,+})}.
\end{equation}
Next, putting in \eqref{2.15} $t=0$, multiplying by $\bar c'_{km}(0)$, summing with respect to $m$ and taking the imaginary part, we find that
\begin{equation}\label{2.19}
\int \Bigl| V_{kt}\big|_{t=0} \Bigr|^2\, dx =0,
\end{equation}
that is $V_{kt}(0,x) \equiv 0$. Then differentiating \eqref{2.15} with respect to $t$, multiplying by $\bar c'_{km}(t)$, summing with respect to $m$ and taking the imaginary part, we find similarly to \eqref{2.18} that 
\begin{equation}\label{2.20}
\|V_{kt}\|_{L_\infty(0,T; L_{2,+})} \leq \|F_t\|_{L_1(0,T;L_{2,+})}.
\end{equation}
Repeating this argument we derive that for any $l$
\begin{equation}\label{2.21}
\|\partial^l_t V_k\|_{L_\infty(0,T; L_{2,+})} \leq \|\partial_t^l F\|_{L_1(0,T;L_{2,+})}.
\end{equation}
Estimate \eqref{2.21} provides existence of a weak solution $V(t,x)$ to problem \eqref{2.13} such that $\partial_t^l V \in C([0,T]; L_{2,+})$ $\forall l$  in the sense of the corresponding integral identity of \eqref{1.4} type. Note that the trace of the function $V|_{t=0} =0$. 
Since according to \eqref{2.13}
\begin{equation}\label{2.22}
V_{xxx} = i a V_{xx} -V_t -b V_x -i F,
\end{equation}
$\partial_t^l V_{xxx} \in C([0,T]; H_+^{-2})$ $\forall l$, that is $\partial_t^l V \in C([0,T];H^1_+)$. Application twice of \eqref{2.22} yields first that $\partial_t^l V_{xxx} \in C([0,T];H_+^{-1})$, that is  $\partial_t^l V\in C([0,T];H^2_+)$, and finally that $\partial_t^l V_{xxx} \in C([0,T];L_{2,+})$, that is $\partial_t^l V\in C([0,T]; H_+^3)$. In particular, the function $V$ satisfies equation \eqref{2.13} a.e. in $\Pi_T^+$ and its trace $V|_{x=0} =0$. For any natural $n$ differentiating equation \eqref{2.13} $3(n-1)$ times with respect to $x$ and using induction with respect to $n$, we derive that $\partial_t^l V \in C([0,T]; H^{3n}_+)$.  

As a result, the function $u \equiv w + V +\Psi$ is the solution of problem \eqref{2.4}, \eqref{2.5} from the space $C^\infty([0,T]; H_+^\infty)$.

Finally, let $\widetilde w(t,x) \equiv u(t,x) \eta(x-1)$. The function $\widetilde w$ solves an initial value problem of \eqref{2.8} type, where the functions $f$, $u_0$ are substituted by corresponding functions $\widetilde f$, $\widetilde u_0$ from the same classes. Then similarly to $w$ the function $\widetilde w \in C^\infty([0,T]; \EuScript S_{exp,+})$ and so
$u \in C^\infty([0,T]; \EuScript S_{exp,+})$.
\end{proof}

Next, consider weak solutions. The notion of a weak solution to the linear problem \eqref{2.4}, \eqref{2.5} is similar to Definition~\ref{D1.1} with obvious modification. In particular in the analog of integral identity  \eqref{1.4} it is sufficient to assume that $u\in L_2(\Pi_T^+)$.

\begin{lemma}\label{L2.3}
A weak solution of problem \eqref{2.4}, \eqref{2.5} is unique in the space $L_2(\Pi_T^+)$.
\end{lemma}

\begin{proof}
This lemma succeeds from the following result on existence of smooth solutions to the corresponding adjoint problem by the standard H\"olmgren argument.
\end{proof}

\begin{lemma}\label{L2.4}
Let $f\in C_0^\infty(\Pi_T^+)$, then there exists a solution $u\in C^\infty([0,T];\EuScript S_+)$ of an initial-boundary value problem in $\Pi_T^+$
\begin{equation}\label{2.23}
i u_t +a u_{xx} - i b u_x - i u_{xxx}  = f(t,x),\quad u(0,x) =0,\quad u(t,0) = u_x(t,0) =0.
\end{equation}
\end{lemma}

\begin{proof}
The scheme of the proof in many ways repeat the proof of Lemma~\ref{L2.2}. First of all, extend the functions $u_0$ and $f$ to the whole real line in a proper way, consider the initial value problem for equation \eqref{2.23} and construct its solution $w \in C^\infty([0,T]; \EuScript S)$ with obvious changes in \eqref{2.9}.

Next, let $\mu_0(t) \equiv - w(t,0)$, $\mu_1(t) \equiv -w_x(t,0)$. Then $\mu_j\in C^\infty[0,T]$ and $\mu_j^{(l)}(0) =0$ $\forall l$. Consider in $\Pi_T^+$ an initial-boundary value problem
\begin{equation}\label{2.24}
i v_t +a v_{xx} -i b v_x-i v_{xxx}  = 0,\quad v(0,x) =0,\ v(t,0) = \mu_0(t), \ v_x(t,0) = \mu_1(t).
\end{equation}
Let $\Psi(t,x) \equiv \mu_0(t) \eta(1-x) + \mu_1(t) x \eta(1-x)$, $F(t,x) \equiv - i\Psi_t - a \Psi_{xx} +ib\Psi_x + i\Psi_{xxx})(t,x)$. It is obvious that $F\in C^\infty([0,T]: \EuScript S_{+})$ and $\partial_t^l F(0,x) \equiv 0$ $\forall l$. Problem \eqref{2.24} is equivalent to a problem
\begin{equation}\label{2.25}
i V_t +a V_{xx} -i b V_x -i V_{xxx}  = F(t,x),\quad V(0,x) =0,\  V(t,0) = V_x(t,0) =0
\end{equation}
for the function $V(t,x) \equiv v(t,x) - \Psi(t,x)$.

Apply the Galerkin method, but here let $\{\varphi_j(x): j=1,2\dots\}$ be a set of linearly independent functions complete in the space $\{\varphi \in H^3_+: \varphi(0) = 0\}$. Construct an approximate solution of problem \eqref{2.25} in the form \eqref{2.14} via conditions
\begin{equation}\label{2.26}
\int \bigl[ i V_{kt}\bar\varphi_m  +  V_k (a \bar\varphi''_{m} +ib \bar\varphi'_{m} + i\bar\varphi'''_{m}) -F\bar\varphi_m \bigr]\, dx =0, \quad m= 1,\dots,k,\ t\in [0,T],
\end{equation}
and \eqref{2.16}.
Multiplying \eqref{2.26} by $2\bar c_{km}(t)$, summing with respect to $m$ and taking the imaginary part, we find that
\begin{equation}\label{2.27}
\frac{d}{dt} \int |V_k|^2\, dx  + \Bigl| V_{kx}\big|_{x=0} \Bigr|^2= 2\Im \int F \bar V_k \,dx,
\end{equation}
whence estimate \eqref{2.18} follows. Next, similarly to \eqref{2.19}, \eqref{2.20} we derive for any $l$ estimate \eqref{2.21}.

Estimate \eqref{2.21} provides existence of a weak solution $V(t,x)$ to problem \eqref{2.25} such that $\partial_t^l V \in C([0,T]; L_{2,+})$ $\forall l$  in the sence of the corresponding integral identity of \eqref{1.4} type and the trace of the function $V|_{t=0} =0$. 
With the use of an equality
$$
V_{xxx} = i F + V_t - i a V_{xx} -b V_x 
$$
Applying the same argument as in the proof of Lemma~\ref{L2.2}, we find that the function $V$ satisfies equation \eqref{2.25} a.e. in $\Pi_T^+$, its traces $V|_{x=0} =V_x|_{x=0} =0$ and $\partial_t^l V \in C([0,T]; H^{3n}_+)$ $\forall l, n$. Therefore, the function $u \equiv w + V +\Psi$ is the solution of problem \eqref{2.23} from the space $C^\infty([0,T]; H_+^\infty)$. Finally, introducing the function $\widetilde w(t,x) \equiv u(t,x) \eta(x-1)$, we find that similarly to $w$ the function $\widetilde w \in C^\infty([0,T]; \EuScript S_{+})$ and so $u \in C^\infty([0,T]; \EuScript S_{+})$.
\end{proof}

\begin{lemma}\label{L2.5}
Let $\psi(x)$ be an admissible weight function, such that $\psi'(x)$ is also an admissible weight function, $u_0\in L_{2,+}^{\psi(x)}$, $\mu\equiv 0$, $f\equiv f_0 +f_{1x}$, where   $f_0\in L_1(0,T;L_{2,+}^{\psi(x)})$, $f_1\in L_{2}(0,T;L_{2,+}^{\psi^{2}(x)/\psi'(x)})$. Then there exist a (unique) weak solution to problem \eqref{2.4}, \eqref{2.5} from the space $X^{\psi(x)}(\Pi_T^+)$ and a function $\mu_1\in L_2(0,T)$, such that for any function $\phi \in C^\infty([0,T]; \EuScript S_+)$, $\phi\big|_{t=T} =0$, $\phi\big|_{x=0} =0$, the following equality holds:
\begin{multline}\label{2.28}
\iint_{\Pi_T^+}\Bigl[u(i \phi_t -a \phi_{xx} +i b \phi_x + i\phi_{xxx})+ f_0 \phi -f_1\phi_x\Bigr]\,dxdt \\
+\int u_0\phi\big|_{t=0}\,dx - i \int_{0}^T \mu_1\phi_{x}\big|_{x=0}\,dt =0.
\end{multline}
Moreover, for $t\in (0,T]$
\begin{multline}\label{2.29}
\|u\|_{X^{\psi(x)}(\Pi_t^+)} +\|\mu_1\|_{L_2(0,t)} \leq c(T)\Bigl(\|u_0\|_{L_{2,+}^{\psi(x)}} +\|f_0\|_{L_1(0,t;L_{2,+}^{\psi(x)})}  \\ 
+\|f_1\|_{L_2(0,t;L_{2,+}^{\psi^2(x)/\psi'(x)})} \Bigr),
\end{multline}
and for a.e. $t\in (0,T)$
\begin{multline}\label{2.30}
\frac{d}{dt}\int |u(t,x)|^2 \psi(x)\,dx  + |\mu_1(t)|^2\psi(0)  + 3\int |u_{x}|^2 \psi'\,dx  = 2a \Im \int u_x \bar u \psi' \,dx \\ + b \int |u|^2 \psi' \,dx  +
\int |u|^2 \psi'''\, dx + 2 \Im \int f_0 \bar u\psi\,dx -2 \Im \int f_1 (\bar u\psi)_x\,dx.
\end{multline}
If $f_1\equiv 0$, then in equality \eqref{2.30} one can put $\psi\equiv 1$.
\end{lemma}

\begin{proof}
First assume that the functions $u_0$ and $f$ satisfy the hypothesis of Lemma~\ref{L2.2} and consider the corresponding smooth solution $u(t,x)$. Let $\psi$ be an admissible weight function. Multiply equation \eqref{2.4} by $2\bar u(t,x) \psi(x)$, integrate over $\mathbb R_+$ and take the imaginary part, then
\begin{multline}\label{2.31}
\int \Im 2iu_t \bar u \psi\,dx + a \int 2\Im u_{xx} \bar u\psi \, dx + b \int \Im 2i u_x \bar u\psi \, dx   + \int \Im 2i u_{xxx} \bar u\psi \, dx  \\= 2\Im \int f \bar u \psi\, dx. 
\end{multline}
Here
$$
\int \Im 2i u_t \bar u\psi \,dx = \int \Re 2 u_t \bar u\psi \, dx = \int (u_t \bar u + u \bar u_t)\psi \,dx = \frac{d}{dt} \int |u|^2\psi \,dx,
$$
$$
\int 2\Im u_{xx} \bar u\psi \,dx = -2\int \Im (|u_x|^2\psi +u_x \bar u \psi')\, dx = -2\Im \int u_x \bar u\psi'\, dx,
$$
$$
\int \Im 2i u_{x}\bar u\psi \,dx = \int 2\Re u_{x} \bar u \psi \, dx = \int (u_{x}\bar u + u \bar u_{x})\psi \, dx  = 
-\int |u|^2\psi' \,dx,
$$
\begin{multline*}
\int \Im 2i u_{xxx}\bar u\psi\,dx = \int 2\Re u_{xxx} \bar u\psi\, dx = \int (u_{xxx}\bar u + u \bar u_{xxx})\psi\, dx \\ = 
-\int (u_{xx} \bar u_x + u_x \bar u_{xx})\psi\, dx - \int (u_{xx}\bar u + u\bar u_{xx})\psi' \,dx \\ =
-\int (|u_x|^2)_x\psi \,dx + 2\int |u_x|^2 \psi'\, dx + \int (|u|^2)_x \psi''\, dx \\ = \bigl(|u_x|^2\psi\bigr)\big|_{x=0} +
3\int |u_x|^2 \psi'\, dx - \int |u|^2 \psi''' \,dx.
\end{multline*}
Therefore, equality \eqref{2.31} provides equality \eqref{2.30} for any $t\in[0,T]$, where $\mu_1 \equiv u_x\big|_{x=0}$. Moreover, in the smooth case equality \eqref{2.28} holds for $\mu_1 \equiv u_x\big|_{x=0}$.

Note, that if $\psi'$ is also an admissible weight function, then for an arbitrary $\varepsilon>0$
\begin{multline}\label{2.32}
\Bigl|\int f_1(\bar u\psi)_x\,dx\Bigr| \leq c \|(|u_{x}|+|u|)(\psi')^{1/2} \|_{L_{2,+}}\|f_1 \psi (\psi')^{-1/2}\|_{L_{2,+}} \\ \leq 
\varepsilon \int \bigl(|u_{x}|^2+|u|^2 \bigr)\psi'\,dx +
c(\varepsilon)\|f_1\|_{L_{2,+}^{\psi^{2}(x)/\psi'(x)}}^{2}.
\end{multline}
Equality \eqref{2.30} and inequality \eqref{2.32} imply that for smooth solutions
\begin{multline}\label{2.33}
\|u\|_{X^{\psi(x)}(\Pi_T^+)} + \|u_{x}\big|_{x=0}\|_{L_2(0,T)}  \\ \leq c(T)\Bigl(\|u_0\|_{L_{2,+}^{\psi(x)}} +\|f_0\|_{L_1(0,T;L_{2,+}^{\psi(x)})}  
+\|f_1\|_{L_2(0,T;L_{2,+}^{\psi^2(x)/\psi'(x)})} \Bigr).
\end{multline}

The end of the proof is performed in a standard way via closure on the basis of estimate \eqref{2.33}.
\end{proof}

Now consider special solutions of homogeneous equation \eqref{2.4} ($f\equiv 0$) or in an equivalent form
\begin{equation}\label{2.34}
u_t - i a u_{xx} + b u_x + u_{xxx} =0.
\end{equation}
These solutions of a ``boundary potential'' type were previously constructed and studied in \cite{F12} for more general equations, so here we present in brief only the main items (formally in that paper the coefficients of the equation were real, but the argument in the complex case is the same). 

The argument is based on the Laplace transform. Consider an ordinary equation
\begin{equation}\label{2.35}
y''' - i a y'' + b y' + i\lambda y =0,\quad \lambda \in \mathbb R \setminus\{0\}.
\end{equation}
Its characteristic algebraic equation is the following:
\begin{equation}\label{2.36}
r^3 - i a r^2 + b r + i \lambda =0.
\end{equation}
If $a=b =0$ there exists the unique root of this equation with the negative real part $r_0 = - (\sqrt3 + i\sgn\lambda) |\lambda|^{1/3}/2$. In the general case similar property holds for large values of $\lambda$, that is there exists $\lambda_0$ (without loss of generality one can assume that $\lambda_0\geq 1$), such that for $|\lambda|\geq \lambda_0$ equaion \eqref{2.36} has a root $r_0$, verifying for certain positive constants $\varepsilon$ and $c$ inequalities
\begin{equation}\label{2.37}
\Re r_0 \leq -2\varepsilon |\lambda|^{1/3},\quad |r_0| \leq c|\lambda|^{1/3}.
\end{equation}
With the use of $r_0$ the special solution on $\mathbb R_+$ of equation \eqref{2.35} $y^+(x;\lambda)$ is constructed such that
\begin{equation}\label{2.38}
\bigl| (y^+)^{(k)}(x;\lambda) \bigr| \leq c(k) \lambda^k e^{-\varepsilon \lambda x}\ \forall k.
\end{equation}  
Let $\mu \in L_2(\mathbb R)$ and $\widehat\mu(\lambda) =0$ for $|\lambda| < \lambda_0$. Define for $t\in\mathbb R$, $x>0$
\begin{equation}\label{2.39}
J^+(t,x;\mu) \equiv \mathcal F^{-1}_t \bigl[ y^+(x;\lambda) \widehat\mu(\lambda) \bigr](t).
\end{equation}
According to \cite{F12} the function $J^+ \in C^\infty(\mathbb R^t \times \mathbb R_+^x)$ and satisfies in this domain equation \eqref{2.34}. Moreover, there exists positive constant $\beta_0$ such that $\forall \beta\in [0,\beta_0)$, $\forall x_0>0$, $\forall k,j$
\begin{equation}\label{2.40}
\sup\limits_{x\geq x_0} e^{\beta x} \|\partial_x^k J^+\|_{H^j(\mathbb R^t)} \leq c(\beta, x_0, k, j) \|\mu\|_{L_2(\mathbb R)},
\end{equation}
\begin{equation}\label{2.41}
\|J^+\|_{C_b(\overline{\mathbb R}_+^x; H^{1/3}(\mathbb R^t))}, \|J^+_x\|_{C_b(\overline{\mathbb R}_+^x; L_2(\mathbb R^t))} \leq c \|\mu\|_{H^{1/3}(\mathbb R)}, \quad
\lim\limits_{x\to +0} J^+(t,x;\mu) = \mu(t),
\end{equation}
\begin{equation}\label{2.42}
\|J^+\|_{C_b(\mathbb R^t; L_{2,+})} \leq c \|\mu\|_{H^{1/3}(\mathbb R)},
\end{equation}
\begin{equation}\label{2.43}
\|J^+\|_{L_2(\mathbb R^t; H^1_+)} \leq c \|\mu\|_{H^{1/6}(\mathbb R)},
\end{equation}
\begin{equation}\label{2.44}
\|J^+\|_{L_2(\mathbb R^t; C_{b,+}^1)} \leq c(s) \|\mu\|_{H^s(\mathbb R)}, \quad s>1/3.
\end{equation}

The boundary potential $J^+$ is used in this paper as a tool to make zero the boundary data at $x=0$. Let $\chi(\lambda)$ be the characteristic function of the interval $(-\lambda_0, \lambda_0)$. For $\mu \in L_2(\mathbb R)$ let
\begin{equation}\label{2.45}
\mu_0(t) \equiv \mathcal F_t^{-1} \bigl[ \widehat\mu(\lambda) \chi(\lambda)\bigr](t),\quad \mu_1(t) \equiv \mu(t) - \mu_0(t).
\end{equation}
For $x>0$ define
\begin{equation}\label{2.46}
\Psi_0(t) \!\equiv \bigl[ \mu_0(t) + J^+(t,x; \mu_1)\bigr] \eta(2-x),  F_0(t,x)\! \equiv (i \Psi_{0t} + a\Psi_{0xx} + i b \Psi_{0x} + 
i \Psi_{0xxx})(t,x).
\end{equation}
Note that $\Psi_0(t,x) = F_0(t,x) =0$ for $x\geq 2$.
\begin{lemma}\label{L2.6}
Let $\mu \in H^{1/3}(0,T$), then $F_0 \in H^\infty(\Pi_T^+)$ and
\begin{gather}\label{2.47}
\Psi_0 \in C([0,T];L_{2,+}) \cap C_b(\overline{\mathbb R}_+;H^{1/3}(0,T)),\quad \lim\limits_{x\to +0} \Psi_0(t,x) = \mu(t), \\
\label{2.48}
\Psi_{0x} \in  C_b(\overline{\mathbb R}_+;L_2(0,T)),
\end{gather}
and if $\mu\in H^s(0,T)$ for $s>1/3$, then in addition
\begin{equation}\label{2.49}
\Psi_0 \in L_2(0,T; C_{b,+}^1).
\end{equation}
\end{lemma}

\begin{proof}
The assertion of the lemma obviously follows from the aforementioned properties of the function $J_+$.
\end{proof}

\section{Existence}\label{S3}

Consider an initial-boundary value problem for an auxiliary equation
\begin{multline}\label{3.1}
i U_t +a U_{xx} + i b U_x +i U_{xxx} +\lambda g(|U+\Psi|) (U+\Psi) +i\beta \bigl(g(|U+\Psi|) (U+\Psi)\bigr)_x \\+ i \gamma \bigl(g(|U+\Psi|)\bigr)_x (U+\Psi) =F,
\end{multline}
where $\Psi = \Psi(t,x)$, $F= F(t,x)$, with initial and boundary conditions
\begin{equation}\label{3.2}
U(0,x) = U_0(x), \quad  U(t,0) =0.
\end{equation}
The notion of a weak solution to this problem is similar to Definition \ref{D1.1}.

\begin{lemma}\label{L3.1}
Let $g\in C^1_{b,+}$, $\psi(x) \equiv e^{2\alpha x}$ for certain $\alpha>0$, $U_0 \in L_{2,+}^{\psi(x)}$, $F\in L_1(0,T;L_{2,+}^{\psi(x)})$, $\Psi\in C([0,T];L_{2,+}^{\psi(x)})$, $\Psi_x \in L_2(0,T;L_{2,+}^{\psi(x)})$. Then there exists $t_0 \in (0,T]$, such that problem \eqref{3.1}, \eqref{3.2} has a unique weak solution $U \in X^{\psi(x)} (\Pi_{t_0}^+)$.
\end{lemma}

\begin{proof}
We apply the contraction principle. For $t_0\in(0,T]$ define a mapping $\Lambda$ on $X^{\psi(x)}(\Pi_{t_0}^+)$ as follows: $U=\Lambda V\in X^{\psi(x)}(\Pi_{t_0}^+)$ is a weak solution to a linear problem
\begin{multline}\label{3.3}
i U_t +a U_{xx} + i b U_x +i U_{xxx} \\ = F -\lambda g(|V+\Psi|) (V+\Psi) -i\beta \bigl(g(|V+\Psi|) (V+\Psi)\bigr)_x - i \gamma \bigl(g(|V+\Psi|)\bigr)_x (V+\Psi)
\end{multline}
in $\Pi_{t_0}^+$ with initial and boundary conditions \eqref{3.2}.

Let
\begin{gather*}
f_1 = f_1(V) \equiv -i (\beta+\gamma) g(|V+\Psi|) (V+\Psi),\\
f_0 = f_0(V) \equiv F - \lambda g(|V+\Psi|) (V+\Psi) + i\gamma g(|V+\Psi|) (V_x +\Psi_x).
\end{gather*}
Note that $\psi^{2}/\psi' \sim \psi$. Then since $|f_1| \leq c (|V| +|\Psi|)$
\begin{multline*}
\|f_1\|_{L_2(0,t_0;L_{2,+}^{\psi^2(x)/\psi'(x)})} \leq c t_0^{1/2} \|f_1\|_{C([0,t_0];L_{2,+}^{\psi(x)})} \\ \leq
c_1 t_0^{1/2} \bigl( \|V\|_{C([0,t_0];L_{2,+}^{\psi(x)})} + \|\Psi\|_{C([0,t_0];L_{2,+}^{\psi(x)})} \bigr)  \leq
c_2 t_0^{1/2} \|V\|_{X^{\psi(x)}(\Pi_{t_0}^+)} + c_2,
\end{multline*}
since $|f_0| \leq c (|F| +|V_x| + |\Psi_x| + |V| + |\Psi|)$
\begin{multline*} 
\|f_0\|_{L_1(0,t_0; L_{2,+}^{\psi(x)})} \leq c \|F\|_{L_1(0,t_0; L_{2,+}^{\psi(x)})} + c t_0^{1/2} \bigl( \|V_x\|_{L_2(0,t_0; L_{2,+}^{\psi(x)})} + 
\|\Psi_x\|_{L_2(0,t_0; L_{2,+}^{\psi(x)})} \\+ \|V\|_{L_2(0,t_0; L_{2,+}^{\psi(x)})} +\|\Psi\|_{L_2(0,t_0; L_{2,+}^{\psi(x)})} \bigr) \leq 
c_1 t_0^{1/2}  \|V\|_{X^{\psi(x)}(\Pi_{t_0}^+)} +c_1.
\end{multline*}
Therefore, Lemma~\ref{L2.5} provides that the mapping $\Lambda$ exists. Moreover, according to inequality \eqref{2.29}
\begin{equation}\label{3.4}
\|\Lambda V\|_{X^{\psi(x)}(\Pi_{t_0}^+)} \leq c + ct_0^{1/2} \|V\|_{X^{\psi(x)}(\Pi_{t_0}^+)}.
\end{equation}
Next, let $V, \widetilde V \in X^{\psi(x)}(\Pi_{t_0}^+)$, $W \equiv V- \widetilde V$.
Note that
$$
|f_1(V) - f_1(\widetilde V)| \leq c \bigl( |V| + |\widetilde V| + |\Psi| +1\bigr) |W|,
$$
where, for example, since $\psi(x) \geq 1$
\begin{multline}\label{3.5}
\|V W\|_{L_2(0,t_0;L_{2,+}^{\psi^2(x)/\psi'(x)})} \leq c \Bigl(\iint_{\Pi_{t_0}^+} |V W|^2 \psi \,dxdt \Bigr)^{1/2}   \\ \leq
c \Bigl( \int_0^{t_0} \sup\limits_{x\geq 0} |V|^2 \,dt \Bigr)^{1/2} \sup\limits_{t\in [0,T]} \Bigl( \int |W|^2 \psi\, dx\Bigr)^{1/2}  \\ \leq
c  \Bigl[ \int_0^{t_0} \Bigl(\int |V_x|^2 \,dx \int |V|^2 \,dx \Bigr)^{1/2} \,dt \Bigr]^{1/2} \|W\|_{X^{\psi(x)}(\Pi_{t_0}^+)}  \\ \leq
c_1 t_0^{1/4} \|V\|_{X^{\psi(x)}(\Pi_{t_0}^+)} \|W\|_{X^{\psi(x)}(\Pi_{t_0}^+)}.
\end{multline}
Therefore,
\begin{multline*}
\|f_1(V) - f_1(\widetilde V)\|_{L_2(0,t_0;L_{2,+}^{\psi^2(x)/\psi'(x)})}  \\ \leq c t_0^{1/4} \bigl( \|V\|_{X^{\psi(x)}(\Pi_{t_0}^+)} + \|\widetilde V\|_{X^{\psi(x)}(\Pi_{t_0}^+)} +1 \bigr)
\|W\|_{X^{\psi(x)}(\Pi_{t_0}^+)}.
\end{multline*}
Note that
$$
|f_0(V) - f_0(\widetilde V)| \leq c|W_x| + c \bigl(|V_x| + |\widetilde V_x| + |\Psi_x| + |V| + |\widetilde V| + |\Psi| +1\bigr) |W|,
$$
where, for example, since $\psi' \sim \psi$
$$
\|W_x\|_{L_1(0,t_0; L_{2,+}^{\psi(x)})} \leq c t_0^{1/2} \|W_x\|_{L_2(0,t_0; L_{2,+}^{\psi'(x)})} \leq c t_0^{1/2} \|W\|_{X^{\psi(x)}(\Pi_{t_0}^+)}
$$
and similarly to \eqref{3.5}
\begin{multline}\label{3.6}
\|V_x W\|_{L_1(0,t_0;L_{2,+}^{\psi(x)})} \leq  \int_0^{t_0} \sup\limits_{x\geq 0} |W| \Bigl( \int |V_x|^2 \psi\, dx\Bigr)^{1/2} \,dt \\ \leq 
\int_0^{t_0} \Bigl(\int |W_x|^2 \,dx \int |W|^2 \,dx \Bigr)^{1/4}  \Bigl( \int |V_x|^2 \psi\, dx\Bigr)^{1/2} \,dt \\ \leq
c t_0^{1/4} \|V\|_{X^{\psi(x)}(\Pi_{t_0}^+)} \|W\|_{X^{\psi(x)}(\Pi_{t_0}^+)}.
\end{multline}
Therefore,
\begin{multline*}
\|f_0(V) - f_0(\widetilde V)\|_{L_1(0,t_0;L_{2,+}^{\psi(x)})}  \\ \leq c t_0^{1/4} \bigl( \|V\|_{X^{\psi(x)}(\Pi_{t_0}^+)} + \|\widetilde V\|_{X^{\psi(x)}(\Pi_{t_0}^+)} +1 \bigr)
\|W\|_{X^{\psi(x)}(\Pi_{t_0}^+)}.
\end{multline*}
As a result, according to inequality \eqref{2.29} 
\begin{equation}\label{3.7}
\|\Lambda V-\Lambda\widetilde{V}\|_{X^{\psi(x)}(\Pi_{t_0}^+)}\leq
c t_0^{1/4}\|V-\widetilde{V}\|_{X^{\psi(x)}(\Pi_{t_0}^+)}
\end{equation}
Assertion of the lemma follows from inequalities \eqref{3.4} and \eqref{3.7}.
\end{proof}

\begin{proof}[Proof of Existence Part of Theorem~\ref{T1.1}]
For $h\in (0,1]$ define functions
\begin{gather}\label{3.8}
g'_h(\theta) \equiv p\theta^{p-1} \eta (2- h\theta),  \quad g_h(\theta) \equiv \int_0^\theta g'_h(y) \,dy, \quad \theta\geq 0,\\
\label{3.9}
f_h(t,x) \equiv f(t,x)\eta(1/h-x), \quad u_{0h}(x)\equiv u_0(x)\eta(1/h-x)
\end{gather}
and consider a set of initial-boundary value problems for an equation
\begin{equation}\label{3.10}
i u_t  + a u_{xx} + i b u_x + i u_{xxx} + \lambda g_h(|u|) u +
 i \beta \bigl( g_h(|u|) u\bigr)_x + i \gamma \bigl( g_h(|u|)\bigr)_x u= f_h,
\end{equation}
with initial and boundary conditions
\begin{equation}\label{3.11}
u\big|_{t=0} = u_{0h}(x), \quad u\big|_{x=0} =0.
\end{equation}
Note that $g_h(\theta)= \theta^p$ if $\theta \leq 1/h$; $0 \leq g_h(\theta) \leq \theta^p$, $g_h(\theta) \leq (2/h)^p$, $0 \leq g'_h(\theta) \leq p\theta^{p-1}$, $g'_h(\theta)  \leq p(2/h)^{p-1}$  $\forall \theta$. 

Lemma~\ref{L3.1}, where $\Psi \equiv 0$, implies that there exist $t_0\in (0,T]$ and a unique solution to this problem $u_h\in X^{e^{2\alpha x}}(\Pi_{t_0}^+)$ for $\alpha = c(1)/2$, where $c(1)$ is the constant from \eqref{1.6}.

Next, establish appropriate estimates on functions $u_h$ uniform with respect to~$h$ (we drop the subscript $h$ in intermediate steps for simplicity). First, note that 
\begin{equation}\label{3.12}
g(|u|)u, g(|u|)u_x, g'(|u|) |u|_x u \in L_1(0,t_0; L_{2,+}^{\psi(x)}) 
\end{equation}
(the last function can be estimated similarly to \eqref{3.6} since $\bigl||u|_x\bigr| \leq |u_x|$, see, for example \cite[Lemma B.1]{F23}),
so the hypothesis of Lemma~\ref{L2.5} is satisfied (for $f_1 \equiv 0$). Then equality \eqref{2.30} provides that for both $\rho(x)\equiv 1$ and $\rho(x) \equiv \psi(x)$
\begin{multline}\label{3.13}
\frac{d}{dt}\int |u|^2 \rho\,dx  + 3\int |u_{x}|^2 \rho'\,dx  \leq  2a \Im \int u_x \bar u \rho' \,dx + b \int |u|^2 \rho' \,dx  \\+
\int |u|^2 \rho'''\, dx + 2 \Im \int \bigl[f -  \lambda g(|u|)u - i\beta \bigl(g(|u|)u\bigr)_x - i\gamma \bigl(g(|u|)\bigr)_x u \bigr]\bar u\rho\,dx.
\end{multline}
Let $\displaystyle g^*(\theta) \equiv \int_0^\theta g(\sqrt{y}) \,dy$ for $\theta\geq 0$ .
Then
\begin{multline}\label{3.14}
-2 \Im i\int \bigl(g(|u|)u\bigr)_x \bar u \rho \,dx = \int g(|u|) (|u|^2)_x \rho \,dx + 2 \int g(|u|) |u|^2 \rho' \,dx \\ =
- \int g^*(|u|^2) \rho' \,dx + 2 \int g(|u|) |u|^2 \rho' \,dx,
\end{multline}
\begin{multline}\label{3.15}
-2 \Im i\int \bigl(g(|u|)\bigr)_x u \bar u \rho \,dx = 2 \int g(|u|) (|u|^2)_x \rho \,dx + 2\int g(|u|) |u|^2 \rho' \,dx \\ =
- 2\int g^*(|u|^2) \rho' \,dx + 2 \int g(|u|) |u|^2 \rho' \,dx.
\end{multline}
Therefore, it follows from \eqref{3.13}--\eqref{3.15} that
\begin{multline}\label{3.16}
\frac{d}{dt}\int |u|^2 \rho\,dx  + 3\int |u_{x}|^2 \rho'\,dx  \leq  2a \Im \int u_x \bar u \rho' \,dx + b \int |u|^2 \rho' \,dx  \\+
\int |u|^2 \rho'''\, dx + 2 \Im \int f \bar u \rho \,dx - (\beta +2\gamma) \!\int g^*(|u|^2) \rho' \,dx +2(\beta +\gamma)\! \int g(|u|) |u|^2 \rho' \,dx.
\end{multline}
Choosing in \eqref{3.16} $\rho\equiv 1$, we obtain, that uniformly with respect to $h$ and $t_0$
\begin{equation}\label{3.17}
\|u_h\|_{C([0,t_0];L_{2,+})} \leq c(T, \|u_0\|_{L_{2,+}}, \|f\|_{L_1(0,T;L_{2,+})}).
\end{equation}
Now choose $\rho(x) \equiv \psi(x)$. Note that uniformly with respect to $h$
\begin{equation}\label{3.18}
\bigl|  g_h^*(|u|^2) \bigr|, g_h(|u|) |u|^2\leq c|u|^{p+2}.
\end{equation}
Let $q=p+2$, $s=s(q)$ from \eqref{2.1}, $\psi_1(x) \equiv \psi'(x)$, $\psi_2(x) \equiv \bigl(\psi'(x)\bigr)^{\frac{2(1-qs)}{q(1-2s)}}$ (note that $qs=p/4<1$). 
Applying interpolating inequality \eqref{2.2}, we obtain that
\begin{multline}\label{3.19}
\int |u|^{p+2}\psi'\,dx = \int |u|^{q} \psi_1^{qs} \psi_2^{q(\frac12-s)} \,dx \\ \leq 
c\Bigl(\int (|u_{x}|^2 + |u|^2)\psi_1 \,dx\Bigr)^{qs} \Bigl(\int |u|^2 \psi_2 \,dx\Bigr)^{q(\frac12-s)} \\ =
c\Bigl(\int (|u_{x}|^2 + |u|^2)\psi' \,dx\Bigr)^{qs} \Bigl(\int (|u|^2\psi')^{\frac{2(1-qs)}{q(1-2s)}} |u|^{\frac{2(q-2)}{q(1-2s)}}\, dx \Bigr)^{q(\frac12-s)} \\ \leq 
c\Bigl(\int (|u_{x}|^2 + |u|^2)\psi' \,dxdy\Bigr)^{p/4} \Bigl(\int |u|^2\psi' \,dx\Bigr)^{(4-p)/4} \Bigl(\int |u|^2 \,dx\Bigr)^{p/2}.
\end{multline}
Since the norm of the functions $u_h$ in the space $L_{2,+}$ is already  estimated in \eqref{3.17}, inequality \eqref{3.16} yields that that uniformly with respect to $h$ and $t_0$
\begin{equation}\label{3.20}
\|u_h\|_{X^{\psi(x)}(\Pi_{t_0}^+)} \leq c(T, \|u_0\|_{L_{2,+}^{\psi(x)}}, \|f\|_{L_1(0,T;L_{2,+}^{\psi(x)})}).
\end{equation}

Note, that repeating the previous argument for $\rho(x) \equiv e^{2\alpha x}$, we obtain that uniformly with respect to $t_0$
\begin{equation}\label{3.21}
\|u_h\|_{X^{e^{2\alpha x}}(\Pi_{t_0}^+)} \leq c(T,h)
\end{equation}
and, therefore, any local solution $u_h$ can be extended to whole time segment $[0,T]$ in the space $X^{e^{2\alpha x}}(\Pi_{T}^+)$ and uniformly with respect to $h$
\begin{equation}\label{3.22}
\|u_h\|_{X^{\psi(x)}(\Pi_{T}^+)} \leq c(T, \|u_0\|_{L_{2,+}^{\psi(x)}}, \|f\|_{L_1(0,T;L_{2,+}^{\psi(x)})}).
\end{equation}

Write the analogue of \eqref{3.16}, where $\rho(x)$ is substituted by $\rho_0(x-x_0)$ for any $x_0\geq 0$. Then it easily follows that (see \eqref{1.9}) uniformly with respect to {h}
\begin{equation}\label{3.23}
\sigma^+ (u_{hx};T) \leq c(T, \|u_0\|_{L_{2,+}}, \|f\|_{L_1(0,T;L_{2,+})}).
\end{equation}

Make one auxiliary calculation. Let $u\in L_\infty(0,T; L_{2,+})$, $\sigma^+(u_x;T) <\infty$, then for any natural $n$ and $Q_{T,n} = (0,T)\times (n-1,n)$  by virtue of \eqref{2.3}  uniformly with respect to $n$ if $1\leq p \leq 5$
\begin{multline}\label{3.24}
\iint_{Q_{T,n}} |u|^{p+1}\,dx dt  \\ \leq 
c \int_0^T\! \Bigl[\Bigl(\int_{n-1}^n\! |u_x|^2\,dx\Bigr)^{(p-1)/4} \!\Bigl(\int_{n-1}^n\! |u|^2\,dx\Bigr)^{(p+3)/4}\,dt  + \Bigl(\int_{n-1}^n\! |u|^2 \,dx\Bigr)^{(p+1)/2}\Bigr] \,dt \\ \leq c \Bigl[T^{(5-p)/4}\|u\|_{L_\infty(0,T;L_{2,+})}^{(p+3)/2}( \sigma^+(u_x;T))^{(p-1)/2} + 
T \|u\|_{L_\infty(0,T;L_{2,+})}^{p+1} \Bigr],
\end{multline}
and if $1\leq p\leq 3$
\begin{multline}\label{3.25}
\iint_{Q_{T,n}} |u|^p |u_x|\,dx dt \leq \Bigl(\iint_{Q_{T,n}} |u|^{2p}\,dx dt\Bigr)^{1/2} \Bigl(\iint_{Q_{T,n}} |u_x|^2\,dx dt\Bigr)^{1/2} \\ \leq
c \Bigl[T^{(3-p)/4}  \|u\|_{L_\infty(0,T;L_{2,+})}^{(p+1)/2} ( \sigma^+(u_x;T))^{(p+1)/2}  + T  \|u\|_{L_\infty(0,T;L_{2,+})}^{p} \sigma^+(u_x;T) \Bigr].
\end{multline}

As a result, in $\Pi_{T,n} = (0,T) \times (0,n) $ for any natural $n$ uniformly with respect to $h$
\begin{equation}\label{3.26}
\|g_h(|u_h|) u_h\|_{L_1(\Pi_{T,n})}, \|\gamma  g_h(|u_h|) u_{hx}\|_{L_1(\Pi_{T,n})} \leq c(n).
\end{equation}
Then from equation \eqref{3.10} itself it follows that uniformly with respect to $h$
\begin{equation}\label{3.27}
\|u_{ht}\|_{L_1(0,T;H^{-2}(0,n))}\leq c.
\end{equation}
Moreover, inequality \eqref{2.3} and estimate \eqref{3.23} provide that uniformly with respect to $h$
\begin{equation}\label{3.28}
\|u_h\|_{L_4(0,T; L_\infty(0,n))} \leq c(n).
\end{equation}
Since the embedding $H^1(0,n) \subset L_\infty(0,n)$ is compact, it follows from \cite[Section 9, Corollary 6]{Sim} that the set $u_h$ is relatively compact in $L_q(0,T;L_\infty(0,n))$ for $q<4$.
Extract a subsequence of the functions $u_h$, again denoted as $u_h$, such that as $h\to +0$
\begin{align*}
u_h\rightharpoonup u&\quad  *-\text{weakly in}\quad L_\infty(0,T; L_{2,+}^{\psi(x)}),\\
u_{hx} \rightharpoonup u_{x}&\quad  \text{weakly in}\quad L_2(0,T;L_{2,+}^{\psi'(x)}),\\
u_h\rightarrow u &\quad \text{strongly in}\quad L_q(0,T;L_\infty(0,n)) \quad \forall n \ \forall q\in [2,4).
\end{align*}
Let $\phi$ be a test function from Definition \ref{D1.1} with $\supp\phi \subset \overline{\Pi}_{T,n}$. According to \eqref{2.28} write the corresponding integral identity of \eqref{1.4}-type for the functions $u_h$ and pass to the limit when $h\to 0$. 
Note that
\begin{equation}\label{3.29}
g_h(|u_h|)u_h -|u|^p u = \Bigl[ g_h(|u_h|)u_h - g_h(|u|)u\Bigr] + \Bigl[g_h(|u|)u - |u|^p u \Bigr],
\end{equation}
where
$$
\bigl|  g_h(|u_h|)u_h - g_h(|u|)u \bigr| \leq c \bigl(|u_h|^4 + |u|^4\bigr) |u_h-u|
$$
and according to \eqref{3.17} and \eqref{3.28}
\begin{multline*}
\bigl\| |u_h|^4 |u_h-u|\bigr\|_{L_1(\Pi_{T,n})} \\ \leq 
c\|u_h-u\|_{L_2(0,T;L_\infty(0,n))} \|u_h\|^2_{L_\infty(0,T;L_{2,+})} \|u_{h}\|^2_{L_4(0,T;L_\infty(0,n))} \to 0.
\end{multline*}
The passage to the limit in the last term from the right-hand side of \eqref{3.29} is obvious since $|u|^p u \in L_1(\Pi_{T,n})$. Next, 
\begin{multline}\label{3.30}
g_h(|u_h|)u_{hx} -|u|^p u_x = \Bigl[ \bigl(g_h(|u_h|) -g_h(|u|)\bigr) u_{hx} \Bigr] +
\Bigl[ \bigl(g_h(|u|) -|u|^p\bigr) u_{hx}\Bigr]  \\ +
\Bigl[ |u|^p (u_{hx} -u_x)\Bigr].
\end{multline}
Here 
$$
\bigl| g_h(|u_h|) -g_h(|u|) \bigr| \leq c \bigl(|u_h|^{p-1} + |u|^{p-1}\bigr) |u_h-u|;
$$
if $2\leq p<3$ then for $q= 4/(4-p)$ (note that $2\leq q<4$)
\begin{multline*}
\bigl\| |u_h|^{p-1} |u_h-u| u_{hx}\|_{L_1(\Pi_{T,n})}  \\ \leq 
c\|u_h-u\|_{L_q(0,T;L_\infty(0,n))} \|u_h\|_{L_\infty(0,T;L_{2,+})}
\|u_h\|^{p-2}_{L_4(0,T;L_\infty(0,n))}  \|u_{hx}\|_{L_2(\Pi_{T,n})}\to 0,
\end{multline*}
if $p< 2$ then
\begin{multline*}
\bigl\| |u_h|^{p-1} |u_h-u| u_{h\,x}\|_{L_1(\Pi_{T,n})}  \\ \leq 
c\|u_h-u\|_{L_2(0,T;L_\infty(0,n))} \bigl(1+ \|u_h\|_{L_\infty(0,T;L_{2,+})}\bigr)
\|u_{hx}\|_{L_2(\Pi_{T,n})} \to 0.
\end{multline*}
The passage to the limit in the last two terms from the right-hand side of \eqref{3.30} is obvious since $|u|^p \in L_2(\Pi_{T,n})$ if $p\leq 3$.
As a result, the function $u \in L_\infty(0,T;L_{2,+}^{\psi(x)})$, $u_x \in L_2(0,T;L_{2,+}^{\psi'(x)})$ and $u$ satisfies integral identity \eqref{1.4} if $\phi(x) =0$ for $x\geq n$. 

Now consider the general case of the test function $\phi$. The well-known embedding $H^2(Q_{T,n}) \subset C_b(\overline{Q}_{T,n})$ yields, that
\begin{multline*}
\sup\limits_{(t,x) \in Q_{T,n}} |\phi(t,x)| \leq c \Bigl(\iint_{Q_{T,n}} \bigl( |\phi_{xx}|^2 + |\phi_{tt}|^2 + |\phi|^2 \bigr) \,dxdt  \\ \leq
 \frac{c_1}{(1+n^2)} \Bigl( \iint_{\Pi_T^+} (1+x^4) \bigl( |\phi_{xx}|^2 + |\phi_{tt}|^2 + |\phi|^2 \bigr) \,dxdt\Bigr)^{1/2}
\end{multline*}
and, therefore,
$$
\sum\limits_{n=0}^{+\infty} \sup\limits_{(t,x) \in Q_{T,n}} |\phi(t,x)| <\infty.
$$
With the use of \eqref{3.24}, \eqref{3.25} we obtain that $|u|^p u \phi, |u|^p u \phi_x, \gamma|u|^p u_x \phi \in L_1(\Pi_T^+)$ and approximating an arbitrary function $\phi$ from Definition \ref{D1.1} by functions, satisfying $\phi(t,x) =0$ for large $x$, and passing to the limit we derive that $u \in X_w^{\psi(x)}(\Pi_T^+)$ is the desired weak solution (the property $u\in C_w([0,T];L_{2,+}^{\psi(x)})$ is established by the standard argument).
\end{proof}

\begin{remark}\label{R3.1}
It follows from \eqref{3.22} that for the constructed weak solution under the hypothesis of Theorem \ref{T1.1}
\begin{equation}\label{3.31}
\|u\|_{X_w^{\psi(x)}(\Pi_{T}^+)} \leq c(T, \|u_0\|_{L_{2,+}^{\psi(x)}}, \|f\|_{L_1(0,T;L_{2,+}^{\psi(x)})}).
\end{equation}
\end{remark}

\begin{proof}[Proof of Existence Part of Theorem~\ref{T1.2}]
As in the proof of the previous theorem for $h\in (0,1]$ introduce functions $g_h$, $f_h$ and $u_{0h}$ by formulas \eqref{3.8}, \eqref{3.9} (note that here $0 \leq g_h(\theta) \leq \theta$)
and consider a set of initial-boundary value problems 
\begin{gather}\label{3.32}
i u_t  + a u_{xx} + i b u_x + i u_{xxx} + \lambda g_h(|u|) u +
 i \beta \bigl( g_h(|u|) u\bigr)_x = f_h, \\
\label{3.33}
u\big|_{t=0} = u_{0h}(x), \quad u\big|_{x=0} = \mu(t).
\end{gather}
Consider the functions $\Psi_0$ and $F_0$, defined in \eqref{2.45}, \eqref{2.46}; let
\begin{equation}\label{3.34}
U(t,x) \equiv u(t,x) - \Psi_0(t,x), \ F_h(t,x) \equiv f_h(t,x) - F_0(t,x), \ U_{0h}(x) \equiv u_{0h}(x) - \Psi_0(x).
\end{equation}
Then problem \eqref{3.32}, \eqref{3.33} is equivalent to a problem
\begin{multline}\label{3.35}
i U_t  + a U_{xx} + i b U_x + i U_{xxx} + \lambda g_h(|U+\psi_0|) (U+\Psi_0)  \\+
 i \beta \bigl( g_h(|U+\Psi_0|) (U+\psi_0)\bigr)_x = F_h, 
\end{multline}
\begin{equation}\label{3.36}
U\big|_{t=0} = U_{0h}(x), \quad U\big|_{x=0} = 0.
\end{equation}
Lemma~\ref{L3.1}, where $\Psi \equiv \Psi_0$, implies that there exist $t_0\in (0,T]$ and a unique solution to this problem $U_h\in X^{e^{2\alpha x}}(\Pi_{t_0}^+)$ for $\alpha = c(1)/2$, where $c(1)$ is the constant from \eqref{1.6}. Let $u_h(t,x) \equiv U_h(t,x) + \Psi_0(t,x)$.

Establish appropriate estimates on functions $u_h$ uniform with respect to~$h$ (we drop the subscript $h$ in intermediate steps for simplicity). Similarly to \eqref{3.13} 
for both $\rho(x)\equiv 1$ and $\rho(x) \equiv \psi(x)$
\begin{multline}\label{3.37}
\frac{d}{dt}\int |U|^2 \rho\,dx  + 3\int |U_{x}|^2 \rho'\,dx  \leq  2a \Im \int U_x \bar U \rho' \,dx + b \int |U|^2 \rho' \,dx  \\+
\int |U|^2 \rho'''\, dx + 2 \Im \int \bigl[F -  \lambda g(|u|)u - i\beta \bigl(g(|u|)u\bigr)_x \bigr]\bar U\rho\,dx.
\end{multline}
Here
\begin{multline}\label{3.38}
-2 \Im i\int \bigl(g(|u|)u\bigr)_x \bar U \rho \,dx = 2 \Re \int g(|u|)u \bar u_x \rho \,dx  -2 \Re \int g(|u|)u \bar \Psi_{0x} \rho \,dx \\ + 2 \int g(|u|) |u|^2 \rho' \,dx -
2 \Re \int g(|u|) u \bar\Psi_0 \rho' \,dx,
\end{multline}
where similarly to \eqref{3.14}
\begin{equation}\label{3.39}
2 \Re \int g(|u|)u \bar u_x \rho \,dx =
- \int g^*(|u|^2) \rho' \,dx -g^*(|\mu|^2) \rho(0),
\end{equation}
\begin{equation}\label{3.40}
- 2\Im \int g(|u|) u \bar U \rho \,dx = 2\Im \int g(|u|)u \bar\Psi_0 \rho \,dx.
\end{equation}
Therefore, it follows from \eqref{3.37}--\eqref{3.40} that
\begin{multline}\label{3.41}
\frac{d}{dt}\int |U|^2 \rho\,dx  + 3\int |U_{x}|^2 \rho'\,dx  \leq  2a \Im \int U_x \bar U \rho' \,dx + b \int |U|^2 \rho' \,dx  \\+
\int |U|^2 \rho'''\, dx + 2 \Im \int F \bar U \rho \,dx - \beta \int g^*(|u|^2) \rho' \,dx - \beta g^*(|\mu|^2) \rho(0) \\ 
-2\beta \Re \int g(|u|)u \bar \Psi_{0x} \rho \,dx +2\lambda \Im \int g(|u|)u \bar\Psi_0 \rho \,dx
+ 2\beta \int g(|u|) |u|^2 \rho' \,dx \\- 2\beta \Re \int g(|u|) u \bar\Psi_0 \rho' \,dx.
\end{multline}
Here for $j=0$ and $j=1$ 
\begin{equation}\label{3.42}
\Bigl| \int g(|u|)u \partial_x^j \bar \Psi_{0} \rho \,dx \Bigr| \leq \sup\limits_{x>0} |\partial_x^j \Psi_{0}| \int |U +\Psi_0|^2 \rho \,dx
\end{equation}
\begin{equation}\label{3.43}
\|g^*(|\mu|^2)\|_{L_2(0,T)} \leq c\|\mu\|_{L_6(0,T)} \leq c_1 \|\mu\|_{H^{1/3}(0,T)}.
\end{equation}
Therefore, choosing in \eqref{3.41} $\rho\equiv 1$, we obtain using the properties of the function $\Psi_0$ from Lemma \ref{L2.6} (in particular, $\Psi_0 \in L_2(0,T; C_{b,+}^1)$), that uniformly with respect to $h$ and $t_0$
\begin{equation}\label{3.44}
\|u_h\|_{C([0,t_0];L_{2,+})} \leq c(T, \|u_0\|_{L_{2,+}}, \|\mu\|_{H^s(0,T)}, \|f\|_{L_1(0,T;L_{2,+})}).
\end{equation}
Choosing $\rho(x) \equiv \psi(x)$ and repeating the argument in \eqref{3.19} for $p=1$ we derive that that uniformly with respect to $h$ and $t_0$
\begin{equation}\label{3.45}
\|u_h\|_{X^{\psi(x)}(\Pi_{t_0}^+)} \leq c(T, \|u_0\|_{L_{2,+}^{\psi(x)}}, \|\mu\|_{H^s(0,T)}, \|f\|_{L_1(0,T;L_{2,+}^{\psi(x)})}).
\end{equation}
The rest part of the proof is the same as for Theorem \ref{T1.1} in the case $p=1$, $\gamma=0$.
\end{proof}

\begin{remark}\label{R3.2}
It follows from \eqref{3.45} that for the constructed weak solution under the hypothesis of Theorem \ref{T1.2}
\begin{equation}\label{3.46}
\|u\|_{X_w^{\psi(x)}(\Pi_{T}^+)} \leq c(T, \|u_0\|_{L_{2,+}^{\psi(x)}}, \|\mu\|_{H^s(0,T)}, \|f\|_{L_1(0,T;L_{2,+}^{\psi(x)})}).
\end{equation}
\end{remark}

\section{Uniqueness and continuous dependence}\label{S4}

\begin{theorem}\label{T4.1}
Let $p\in [1,2]$, $\psi(x)$ be an admissible weight function, such that $\psi'(x)$ be also an admissible weight function and inequality \eqref{1.10} be verified. Then for any $T>0$ and $M>0$ there exists a constant $c=c(T,M)$, such that for any two weak solutions $u(t,x,y)$ and $\widetilde u(t,x,y)$ to problem \eqref{1.1}--\eqref{1.3}, satisfying $\|u\|_{X_w^{\psi(x)}(\Pi_T^+)}, \|\widetilde u\|_{X_w^{\psi(x)}(\Pi_T^+)} \leq M$, with corresponding data $u_0, \widetilde u_0\in L_{2,+}^{\psi(x)}$, $\mu, \widetilde\mu \in H^{1/3}(0,T)$, $f, \widetilde f\in L_1(0,T;L_{2,+}^{\psi(x)})$ the following inequality holds:
\begin{equation}\label{4.1}
\|u -\widetilde u\|_{X_w^{\psi(x)}(\Pi_T^+)} \leq c\bigl( \|u_0 - \widetilde u_0\|_{L_{2,+}^{\psi(x)}} + \|\mu - \widetilde\mu\|_{H^{1/3}(0,T)} +
\|f-\widetilde f\|_{L_1(0,T;L_{2,+}^{\psi(x)})}\bigr).
\end{equation} 
\end{theorem}

\begin{proof}
Let the functions $\Psi_0$ and $F_0$ be defined by formulas \eqref{2.45}, \eqref{2.46} the functions $\widetilde\Psi_0$ and $\widetilde F_0$ in a similar way for $\widetilde\mu$ and $\Psi\equiv \Psi_0-\widetilde\Psi_0$. Then, in particular,
\begin{equation}\label{4.2}
\|\Psi\|_{X^{\psi(x)}(\Pi_T^+)} \leq c(T)\|\mu-\widetilde\mu\|_{H^{1/3}(0,T)}.
\end{equation}
Let $w(t,x) \equiv u(t,x) -\widetilde u(t,x)$, $W_0\equiv u_0-\widetilde u_0 - \Psi\big|_{t=0}$, $F\equiv f-\widetilde f -(F_0 - \widetilde F_0)$, then
\begin{gather}\label{4.3}
\|W_0\|_{L_{2,+}^{\psi(x)}} \leq \|u_0-\widetilde u_0\|_{L_{2,+}^{\psi(x)}} +
c(T)\|\mu-\widetilde\mu\|_{H^{1/3}(0,T)}, \\
\label{4.4}
\|F\|_{L_1(0,T;L_{2,+}^{\psi(x)})} \leq \|f-\widetilde f\|_{L_1(0,T;L_{2,+}^{\psi(x)})} +
c(T)\|\mu-\widetilde\mu\|_{H^{1/3}(0,T)}.
\end{gather}
Let 
\begin{gather*}
f_0 \equiv F - \lambda \bigl(|u|^p u - |\widetilde u|^p \widetilde u\bigr) + i\gamma \bigl(|u|^p u_x - |\widetilde u|^p \widetilde u_x\bigr), \\
f_1 \equiv -i(\beta + \gamma) \bigl( |u|^p u - |\widetilde u|^p \widetilde u\bigr).
\end{gather*} 
The function $W(t,x) \equiv  w(t,x) -\Psi(t,x)$ is a weak solution to an initial-boundary value problem in $\Pi_T^+$ for an equation
\begin{equation}\label{4.5}
i W_t +a W_{xx} +ib W_x + i W_{xxx} = f_0 + f_{1x}
\end{equation}
\begin{equation}\label{4.6}
W\big|_{t=0} =W_0,\qquad W\big|_{x=0}=0.
\end{equation}

Apply for the function $W$ Lemma~\ref{L2.5}. Note that inequality \eqref{1.10} implies that 
\begin{equation}\label{4.7}
(\psi/\psi')^{1/2} \leq c(\psi' \psi)^{p/4}
\end{equation}
and, therefore, according to \eqref{2.2}
\begin{multline}\label{4.8}
\Bigl(\int |u|^{2p} |u_x|^2 \psi\, dx\Bigr)^{1/2}   \leq
\bigl\| u (\psi/\psi')^{1/(2p)}\bigr\|^p_{L_{\infty,+}} \Bigl(\int |u_x|^{2} \psi'\,dx\Bigr)^{1/2} \\ \leq 
c \bigl\| u (\psi')^{1/4} \psi^{1/4}\bigr\|^p_{L_{\infty,+}} \bigl\|u_x (\psi')^{1/2}\bigr\|_{L_{2,+}} \\ \leq
c_1 \Bigl( \int (|u_{x}|^2 + |u|^2)\psi' \,dx\Bigr)^{(p+2)/4} \Bigl(\int |u|^2 \psi\, dxdy\Bigr)^{p/4},
\end{multline}
\begin{multline}\label{4.9}
\int |u|^{2(p+1)} \psi^2/\psi' \, dx \leq 
\bigl\| u (\psi/\psi')^{1/(2p)}\bigr\|^{2p}_{L_{\infty,+}} \int |u|^2 \psi \,dx \\ \leq
c \bigl\| u (\psi')^{1/4} \psi^{1/4}\bigr\|^{2p}_{L_{\infty,+}} \bigl\|u \psi^{1/2}\bigr\|^2_{L_{2,+}} \\ \leq
c_1 \Bigl( \int (|u_{x}|^2 + |u|^2)\psi' \,dx\Bigr)^{p/2} \Bigl(\int |u|^2 \psi\, dxdy\Bigr)^{(p+2)/2},
\end{multline}
so $|u|^p u_x \in L_1(0,T;L_{2,+}^{\psi(x)})$, $|u|^p u \in L_2(0,T; L_{2,+}^{\psi^2(x)/\psi'(x)})$ since $p\leq 2$.
As a result, we derive from \eqref{2.30} that for $t\in (0,T]$
\begin{multline}\label{4.10}
\frac{d}{dt}\int |W|^2 \psi\,dx  + 3\int |W_{x}|^2 \psi'\,dx  \leq 2a \Im \int W_x \bar W \psi' \,dx \\ + b \int |W|^2 \psi' \,dx  +
\int |W|^2 \psi'''\, dx + 2 \Im \int f_0 \bar W\psi\,dx -2 \Im \int f_1 (\bar W\psi)_x\,dx.
\end{multline}

Note that 
\begin{gather}\label{4.11}
\bigl| |u|^p u_x - |\widetilde u|^p \widetilde u_x \bigr| \leq c |u|^p |w_x| + c |\widetilde u_x| \bigl(|u|^{p-1} + |\widetilde u|^{p-1}\bigr) |w|, \\
\label{4.12}
\bigl| |u|^p u - |\widetilde u|^p \widetilde u \bigr| \leq c \bigl(|u|^p + |\widetilde u|^p\bigr) |w|.
\end{gather}
Then similarly to \eqref{4.7}, \eqref{4.8} 
\begin{multline}\label{4.13}
\int |\widetilde u_x| |u|^{p-1} |w W| \psi \,dx \leq \Bigl( \int |u|^{2(p-1)} |W|^2 |\widetilde u_x|^2 \psi \,dx\Bigr)^{1/2} \Bigl(\int |w|^2 \psi \,dx\Bigr)^{1/2} \\ \leq
c\|u (\psi')^{1/4} \psi^{1/4}\|^{p-1}_{L_{\infty,+}} \|W (\psi')^{1/4} \psi^{1/4}\|_{L_{\infty,+}} \|\widetilde u_x (\psi')^{1/2}\|_{L_{2,+}} \|w \psi^{1/2} \|_{L_{2,+}} \\ \leq
c_1 \Bigl(\int (|u_x|^2 + |\widetilde u_x|^2 + |u|^2 + |\widetilde u|^2) \psi' \,dx\Bigr)^{(p+1)/4} \Bigl( \int (|u|^2 + |\widetilde u|^2) \psi\,dx\Bigr)^{(p-1)/4} \\ \times
\Bigl( \int (|W_x|^2 + |W|^2)\psi' \,dx\Bigr)^{1/4} \Bigl( \int (|W|^2 + |\Psi|^2) \psi \,dx\Bigr)^{3/4}  \leq \varepsilon \int |W_x|^2 \psi' \,dx  \\ + c(\varepsilon) 
\Bigl[\Bigl(\int (|u_x|^2 + |\widetilde u_x|^2 + |u|^2 + |\widetilde u|^2) \psi' \,dx\Bigr)^{(p+1)/3} \Bigl( \int (|u|^2 + |\widetilde u|^2) \psi\,dx\Bigr)^{(p-1)/3} +1\Bigr] \\ \times
\int  (|W|^2 + |\Psi|^2) \psi \,dx \equiv \varepsilon \int |W_x|^2 \psi' \,dx + \omega(t) \int  (|W|^2 + |\Psi|^2) \psi \,dx,
\end{multline}
where $\varepsilon>0$ can be chosen arbitrarily small and $\omega \in L_1(0,T)$ since $p\leq 2$. Next, similarly to \eqref{4.7}, \eqref{4.9}
\begin{multline}\label{4.14}
\int |u|^p |w W_x| \psi \,dx \leq \Bigl( \int |u|^{2p} |w|^2 \psi^2/\psi' \,dx \Bigr)^{1/2} \Bigl( \int |W_x|^2 \psi' \,dx\Bigr)^{1/2} \\ \leq
c \|u (\psi')^{1/4} \psi^{1/4}\|^p_{L_{2,+}} \|w \psi^{1/2}\|_{L_{2,+}} \|W_x (\psi')^{1/2} \|_{L_2,+} \\ \leq
c_1 \Bigl(\int (|u_x|^2 + |u|^2)  \psi' \,dx\Bigr)^{p/4} \Bigl( \int |u|^2 \psi\,dx\Bigr)^{p/4} \Bigl( \int (|W|^2 + |\Psi|^2)\psi \,dx \Bigr)^{1/2} \\ \times
 \Bigl( \int |W_x|^2 \psi' \,dx\Bigr)^{1/2} \leq \varepsilon \int |W_x|^2 \psi' \,dx \\ +
 c(\varepsilon) \Bigl(\int (|u_x|^2 + |u|^2)  \psi' \,dx\Bigr)^{p/2} \Bigl( \int |u|^2 \psi\,dx\Bigr)^{p/2} \int (|W|^2 + |\Psi|^2)\psi \,dx \\ \equiv
 \varepsilon \int |W_x|^2 \psi' \,dx + \omega(t) \int  (|W|^2 + |\Psi|^2) \psi \,dx,
\end{multline}
where again $\varepsilon>0$ can be chosen arbitrarily small and $\omega \in L_1(0,T)$ since $p\leq 2$.
Then inequalities \eqref{4.2}--\eqref{4.4}, \eqref{4.10}--\eqref{4.14} provide the desired result.
\end{proof}

\begin{proof}[Proof of Uniqueness Part of Theorem~\ref{T1.1}]
The result on uniqueness and continuous dependence in Theorem~\ref{T1.1} follows from Theorem~\ref{T4.1} (where $\mu = \widetilde\mu \equiv 0$) and estimate \eqref{3.31}.
\end{proof}

\begin{proof}[Proof of Uniqueness Part of Theorem~\ref{T1.2}]
The result on uniqueness and continuous dependence in Theorem~\ref{T1.2} follows from Theorem~\ref{T4.1} (where $p=1$) and estimate \eqref{3.46}.
\end{proof}

\section*{Compliance with Ethical Standards}

The author declares that there is no conflict of interest. 






\end{document}